\documentclass[11pt]{amsart}

\usepackage[utf8]{inputenc}

\usepackage{eucal}
\usepackage{upgreek}
\usepackage{pdfsync}
\usepackage[all, cmtip]{xy}
\usepackage{amsfonts}
\usepackage{amsmath}
\usepackage{amssymb}
\usepackage{amscd}
\usepackage{color,xcolor}
\usepackage{amsthm}
\usepackage{amsxtra}
\usepackage{bm}
\usepackage{bbm}
\usepackage{graphicx}
\usepackage[hmargin=3cm,vmargin=3cm, voffset=-25pt, footskip = 50pt]{geometry}
\usepackage{mathrsfs}
\usepackage[numbers,sort]{natbib}
\usepackage{isomath}
\usepackage{enumitem}
\usepackage{latexsym} 
\usepackage{dsfont}
\usepackage{url}
\usepackage{mathtools}
\usepackage{tikz-cd}
\usepackage{parskip}

\newtheorem{thm}{Theorem}

\newtheorem{prop}{Proposition}
\newtheorem{lma}[prop]{Lemma}

\newtheorem{clm}[prop]{Claim}

\theoremstyle{definition}

\theoremstyle{remark}

\newtheorem{rmk}[prop]{Remark} 

\newtheorem*{qtn*}{Question}

\def\mrm#1{{\mathrm{#1}}}

\def\cl#1{{\mathcal{#1}}}

\newcommand{\R}{{\mathbb{R}}}
\newcommand{\Z}{{\mathbb{Z}}}

\newcommand{\OP}{\operatorname}

\newcommand{\B}[1]{{\mathbf #1}}
\newcommand\vol{\operatorname{vol}}

\newcommand\area{\operatorname{area}}

\newcommand{\G}{\mathcal{G}}

\newcommand{\Id}{{\mathbbm{1}}}

\newcommand{\ol}[1]{\overline{#1}}

\renewcommand{\geq}{\geqslant}
\renewcommand{\leq}{\leqslant}

\DeclareMathOperator{\Ham}{\mathrm{Ham}}

\def\H2{H^{(2)}}

\usepackage[hyperindex]{hyperref}

\newcommand{\esemail}{egor.shelukhin@umontreal.ca}
\newcommand{\mbemail}{brandens@bgu.ac.il}

\begin{document}



\title[The $L^p$-diameter of the space of contractible loops]{The $L^p$-diameter of the space of contractible loops}

\author[Michael Brandenbursky and Egor Shelukhin]{Michael Brandenbursky and Egor Shelukhin}\thanks{MB was partially supported by the Israel Science Foundation grant 823/23.\\ \indent ES was partially supported an NSERC Discovery grant and NSF Grant No. DMS-1926686.}


\address{Department of Mathematics, Ben Gurion University of the Negev, Israel}
\email{\mbemail}

\address{Department of Mathematics and Statistics, University of Montreal, C.P. 6128 Succ.  Centre-Ville Montreal, QC H3C 3J7, Canada}
\email{\esemail}

\begin{abstract}
We prove that the space of contractible simple loops of a given fixed area in any compact oriented surface has infinite diameter as a homogeneous space of the group of area-preserving diffeomorphisms endowed with the $L^p$-metric.
As a special case, this resolves the $L^p$-metric analogue of the well-known question in symplectic topology regarding the space of equators on the two-sphere. Our methods involve a new class of functionals on a normed group, which are more general than quasi-morphisms.
\end{abstract}	


\maketitle

\section{Introduction and main results}

Consider the two-sphere $S^2$ with the standard area form and the space $\cl E$ of equators in it: smooth embedded loops dividing the sphere into two components of equal areas. The identity component $\cl G$ of the group area-preserving diffeomorphisms acts transitively on $\cl E.$ This allows us to induce a pseudo-metric on $\cl E$ given a metric on $\cl G.$ A remarkable bi-invariant such metric on $\cl G$ was introduced by Hofer \cite{Hofer-90} and subsequently Viterbo \cite{Viterbo-generating}, Polterovich \cite{Polterovich-Hofer}, and Lalonde-McDuff \cite{LM-energy}. The induced pseudo-metric on $\cl E$ is non-degenerate, and is therefore a metric, as was proven by Chekanov \cite{Chek-finsler}. The first interesting invariant of a metric space is its diameter. Polterovich \cite{P-sphere} proved that the diameter of $\cl G$ in Hofer's metric in infinite, a result which was recently extended to provide quasi-isometric embeddings of linear spaces of arbitrarily large dimension \cite{CHS-sphere, PS-sphere}. However, the diameter of $\cl E$ in the Chekanov-Hofer metric is still unknown at the time of writing, see \cite[Problem 32]{MS}.

In this paper we show that the diameter of $\cl E$ is infinite in the metric induced from the right-invariant $L^p$-metric on $\cl G.$ The case $p=2$ is closely related to questions in hydrodynamics, while the case $p=1$ has the following intuitive definition (see \cite{BS-sphere}). The $L^1$-length $\ell_1(\{\phi_t\})$ of an isotopy $\{\phi_t\}$ is the average length of a trajectory $\{\phi_t(x)\}$ of a point $x,$ and the $L^1$-norm of $\phi$ is defined as \[|\phi|_1 = \inf_{\phi_1 = \phi} \ell_1(\{\phi_t\}).\] It is easy to see that this norm is non-degenerate and defines the right invariant metric 
$$d_1(\phi, \psi) = | \psi \phi^{-1} |_1.$$ 
This metric induces a pseudo-metric on $\cl E,$ given by \[d_1(L,L') = \inf_{\phi(L) = L' } |\phi|_1.\] This pseudo-metric is non-degenerate as shown in Proposition \ref{prop: non-deg} below, and therefore defines a metric.  Our main result concerning the resulting metric space is as follows.

\begin{thm}\label{thm: diam eq}
The diameter of $\cl E$ equipped with the metric $d_1$ is infinite.
\end{thm}

It is easy to see (see for example \cite[Remark 1.1]{BS-sphere}) that 
in this case the diameter of $\cl E$ in $d_p$ for all $p \geq 1$ is also infinite. 

Moreover, using the same key ideas, but with additional technical modifications, we generalize this result as follows. 
Let $\Sigma$ be a closed oriented surface 
of positive genus with an area form $\omega$. Let $\G_\Sigma$ be 
the identity component of the group of area-preserving diffeomorphisms of $\Sigma$.

\begin{thm}\label{thm: generalized}
Let $\Sigma$ be a closed oriented surface and $L \subset \Sigma$ a smooth embedded separating curve. Then the orbit $\cl E_L$ of $L$ under the action of $\cl G_\Sigma$ equipped with the metric $d_1$ has infinite diameter. 
\end{thm}

\begin{rmk}
 Let $\Ham(\Sigma)<\G_\Sigma$ the group of Hamiltonian diffeomorphisms of $\Sigma$. 
 Theorem \ref{thm: generalized} implies that
 the diameter of $\cl E_L$ equipped with the $d_1$-metric coming from $\Ham(\Sigma)$ is infinite.
\end{rmk}

As a special case, this implies that the space $\cl E_a$ of all embedded loops bounding a {\em disk} of area 
$a \in (0, \mathrm{Area}(\Sigma))$ in every closed oriented surface $\Sigma$ with an area form (see also Remark \ref{rmk: compact}) with $d_1$ has infinite diameter. We call these loops $a$-equators. In fact, for $\Sigma = S^2,$ the argument proving Theorem \ref{thm: diam eq} extends verbatim to $\cl E_a,$ with the 
caveat that for the sphere $\cl E_a = \cl E_{A-a}$ where $A = \mrm{Area}(S^2).$

Note that the definition of the $L^1$-metric depends on the choice of an auxiliary Riemannian metric on the surface. However, our main results are independent of this choice. The same remark applies to all the $L^p$-metrics. Furthermore, in this paper we consider unparametrized embedded loops, however it is immediate from the proof of Lemma \ref{lma: frag} that all results apply to the parametrized case as well.

The methods of the proof involve functionals obtained from quasi-morphisms on braid groups in the spirit of \cite{GG-commutators}, see for instance \cite{BS-sphere, BMS-SM}. However, our functionals are {\em not} quasi-morphisms on $\cl G$ as it has usually been in previous work in the field. Instead, they satisfy weaker properties, which, however, are sufficient for our purposes. Thus, we introduce the notion of a {\em para-morphism} on a metric group.
We now outline our approach in slightly greater detail.

\subsection{Outline of the proof}

First, we show, using integration along the configuration space of points in $S^2,$ and crucially \cite{BMS-SM} for the first point, the following result.

\begin{prop}\label{prop: prop}
For each equator $L$ there is a functional $\Phi:\cl G\to\R$ 
satisfying the following properties for $A, B,C,D$ depending only on $L$ and the geometry of $S^2$:\\
\begin{enumerate}
\item\label{p1} $|\Phi(gf) - \Phi(f)-\Phi(g)|\leq C + D |g|_1$ for all $f,g \in \cl G$
\item\label{p2} $\ol{\Phi}(f) = \liminf_{k \to \infty} \Phi(f^k)/k$ is well-defined on $\cl G$ and not identically zero
\item\label{p3} $|\Phi(h)| \leq B$ for all $h \in \cl G$ such that $h(L) = L,$ and in particular $\ol{\Phi}(h)=0$
\item\label{p4} $|\Phi(f)| \leq A (|f|_1 + 1)$ for all $f \in \cl G$
\end{enumerate}
\end{prop}

\begin{rmk}Note that if $D = 0$ then $\Phi$ is a quasi-morphism on $\cl G.$ We therefore call a functional satisfying the properties \ref{p1}, \ref{p3}, \ref{p4} an $L^1$ {\em para-morphism}. In addition, Property \ref{p1} implies that $\delta\Phi$  is a $1$-bounded group two-cocycle and defines a certain cohomology class, see 
\cite{Gal-Kedra-2cocycle, Gal-Kedra-refinement}. The notion of a para-morphism can be generalized to the general framework on a normed group $(G,\nu),$ where $\nu$ is the norm, and we expect it to yield further applications in geometry and dynamics. We remark that we construct such $\Phi$ using surface braid groups on $n$ strands for each $n>3$ and denote it by $\Phi_n.$ \end{rmk}

We now explain how Proposition \ref{prop: prop} implies Theorem \ref{thm: diam eq}.

\begin{proof}[Proof of Theorem \ref{thm: diam eq}]

Let $L$ be the standard equator and $\Phi$ provided by Proposition \ref{prop: prop}. Now suppose that $g(L) = f(L)$ for $f, g \in \cl G.$ Then $h = g^{-1} f$ satisfies $h(L) = L$. Therefore, 
$$|\Phi(h)| \leq B$$ 
by Property \ref{p3}. From Properties \ref{p1} and \ref{p4} we obtain
\[|\Phi(f)| \leq |\Phi(h)| + |\Phi(g^{-1})| + C+ D|g^{-1}|_1 \leq A+B+C + (A+D) |g|_1.\] Taking infimum over all $g \in \cl G$ with $g(L) = f(L)$ yields \begin{equation}\label{eq: bound} |\Phi(f)| \leq E + F\, d_1(L,f(L))\end{equation} 
for constants 
$$E = A+B+C, F = A+B.$$ 
Now Property \ref{p2} implies that there is a sequence $f_i \in \cl G$ such that 
$$|\Phi(f_i)| \to \infty$$
as $i \to \infty$ and hence by \eqref{eq: bound} 
$$d_1(L, f_i(L)) \to \infty$$ 
as well. This finishes the proof. \end{proof}

\subsection{Organization of the paper}
In Section \ref{sec: prop} we prove Proposition \ref{prop: prop} and in Section \ref{sec: other} we explain how the arguments generalize to prove an analogue of this proposition for $\cl E_L,$ where $\Sigma$ is a closed oriented surface and $L \subset \Sigma$ is an embedded separating curve and hence prove Theorem \ref{thm: generalized}.

\section*{Acknowledgments}
MB was partially supported by the Israeli Science Foundation grant 823/23 and CRM-Simons fellowship. 
MB wishes to express his gratitude to CRM for the support and excellent working conditions.

ES was partially supported by an NSERC Discovery grant, a Courtois chair in fundamental research, a Fonds de Recherche du Qu\'ebec Nature et Technologies teams grant, and a Sloan Research Fellowship. This work was also partially supported by the National Science Foundation
under Grant No. DMS-1926686. ES thanks the Institute for Advanced Study for an excellent research atmosphere.

\section{Preliminaries}
Recall that an alternative way to define the $L^1$-length is \[\ell_1(\{\phi_t\}) = \int_0^1\int_\Sigma |X_t| \omega\,dt\] where $X_t$ is the time-dependent vector field generating $\phi_t$ and $\omega$ is the area form on $\Sigma$. 
Note that for a constant $C_0$ depending only on the geometry of $\Sigma$ we have 
 \[\ell_1(\{\phi_t\}) \leq C_0 \int_0^1 \int_\Sigma |d(H_t)| \omega\,dt,\] where $H_t(x) = H(t,x)$ is the time-dependent Hamiltonian generating $X_t$.
 
\begin{prop}\label{prop: non-deg} Let $L$ be an embedded separating loop on a compact orientable surface $\Sigma$. Let $\cl E_L$ be the orbit of $L$ under the identity component $\cl G$ of the group of area-preserving diffeomorphisms.
Then the metric $d_1$ on $\cl E_L$ induced by the $L^1$ metric on $\cl G$ is non-degenerate.
\end{prop}

By \cite[Remark 1.1]{BS-sphere}, Proposition \ref{prop: non-deg} applies to all $L^p$ metrics for $p \geq 1$.

\begin{proof}
Let $A$ and $B$ be the two connected components of the complement of $L.$ Then $\phi L \neq L$ 
if and only if $\phi(A) \cap B \neq \emptyset.$ Let $U$ be a connected component of 
the non-empty open set $\phi(A)\cap B$ and $D = D(z, \varepsilon)$ 
be a metric ball of radius $\varepsilon > 0$ in $U,$ such that $D(z, 2\varepsilon)$ is also in $U$. 
Note that $U \subset B$ while $\phi^{-1} U \subset A.$ 
Therefore, $\phi^{-1}(U) \cap U = \emptyset$ and moreover $d(\phi^{-1}x, x) > \varepsilon$ 
for all $x \in D.$ Thus 
$$|\phi|_1 = |\phi^{-1}|_1 \geq \varepsilon \cdot \mrm{Area}(D)$$ 
and hence taking infimum over $\phi$ with $\phi(L) = L'$ we get 
$$d_1(L', L)  \geq \varepsilon \cdot \mrm{Area}(D) > 0.$$
\end{proof}

\section{Proof of Proposition \ref{prop: prop}}\label{sec: prop}
Without loss of generality let $L$ be the standard equator. 
We proceed in a number of steps starting with a definition of the functional $\Phi_n$.

\subsection{Definition of the para-morphism}\label{ssec: paramorphism}
Let $n\in\mathbb{N}$ such that $n>3$, and let $\OP{C}_n(S^2)$ 
be the configuration space of all unordered $n$-tuples of pairwise
distinct points in $S^2$. Recall that the Birman map (see e.g. \cite{Farb-Margalit}):
$$
\OP{Push}\colon B_n(S^2)\to\OP{MCG}(S^2,n),
$$
where 
$$B_n(S^2)=\pi_1(\OP{C}_n(S^2),z)$$ 
is the spherical braid group on $n$ strands
and $\OP{MCG}(S^2,n)$ is the mapping class group of the $n$-punctured sphere, 
is defined as follows: 
let $\alpha(t)$, $t\in[0,1]$, be a loop in $\OP{C}_n(S^2)$ based at $z$ 
and $h_t\in\OP{Diff}(S^2)$ an isotopy such that $h_t(z)=\alpha(t)$. 
We define $\OP{Push}(\alpha):=[h_1]$ where $\alpha$ is the braid represented by the loop $\alpha(t)$.
The braid $\alpha$ is called \emph{reducible} if $\OP{Push}(\alpha)$ is a reducible mapping class.
We denote by $Q_{\OP{BF}}(B_n(S^2))$ the space of homogeneous quasimorphisms
on $B_n(S^2)$ which vanish on reducible braids.
It follows from the celebrated paper by Bestvina and Fujiwara \cite{BestvinaFujiwara} that
the space $Q_{\OP{BF}}(B_n(S^2))$ is infinite dimensional, see \cite[Section 4.A.]{2019-BM1}.

Let $\{f_t\}$ be an isotopy in $\G$ from the identity to $f\in\G$. For $x,y\in S^2$ we define a loop
$\gamma_{x,y}\colon [0,1]\to S^2$ by
$$
\gamma_{x,y}(t):=
\begin{cases}
\alpha_{3t} &\text{ for } t\in \left [0,\frac13\right ]\\
f_{3t-1}(x) &\text{ for } t\in \left [\frac13,\frac23\right ]\\
\beta_{3t-2} & \text{ for } t\in \left [\frac23,1\right ],
\end{cases}
$$
where $\alpha_t$ is a shortest path on $S^2$ from $y$ to $x$,
and $\beta_t$ is a shortest path on $S^2$ from $f(x)$ to $y$.

Let $\OP{X}_n(S^2)$ be the configuration space of all ordered $n$-tuples
of pairwise distinct points in the two-sphere $S^2$. Its fundamental group
$\pi_1(\OP{X}_n(S^2),z)$ is identified with the spherical pure braid group $P_n(S^2),$ where $z=(z_1,\ldots,z_n)$ in $\OP{X}_n(S^2)$ is a base point.
For almost every $x=(x_1,\ldots,x_n)\in \OP{X}_n(S^2)$ the
$n$-tuple of loops $(\gamma_{x_1,z_1},\ldots,\gamma_{x_n,z_n})$ is
a based loop in the configuration space $\OP{X}_n(S^2)$.
Let 
$$\gamma(\{f_t\},x)\in P_n(S^2)=\pi_1(\OP{X}_n(S^2),z)$$
be the element represented by this loop, and let
$$\varphi\colon P_n(S^2)\to \B R$$ 
be a homogeneous quasimorphism.
Since the group $\pi_1(\OP{Ham}(S^2))$ is isomorphic to $\Z/2\Z$, the number
$\varphi(\gamma(\{f_t\},x))$ does not depend on the choice of the isotopy $\{f_t\}$ and from now on
is denoted by $\varphi(\gamma(f,x))$. Note that for $f,g\in\G$ we have the following cocycle condition:
\begin{equation}\label{eq:cocycle}
\varphi(\gamma(gf,x))=\varphi(\gamma(g,f(x))\cdot\gamma(f,x)).
\end{equation}
Denote by $D_+$ and $D_-$ the open Northern and Southern hemisphere respectively. Consider the subspace $\OP{X}_n(D_+\cup D_-)\subset\OP{X}_n(S^2)$
consisting of those $n$-tuples points in $S^2$ where either all points lie in $D_+,$ or all points lie in $D_-$, or exactly $n-1$ points lie in $D_+$ and one point in $D_-$, or exactly $n-1$ points lie in $D_-$ and one point in $D_+$. Let $\varphi\in Q_{\OP{BF}}(B_n(S^2))$ a nontrivial homogeneous quasimophism. We define 
$$\Phi_n: \cl G \to \R$$
as follows:
\begin{equation}\label{eq:paramorphism}
\Phi_n(f):=
\int\limits_{\OP{X}_n(S^2)}\varphi(\gamma(f,{x}))d{x}\quad-\int\limits_{\OP{X}_n(D_+\cup D_-)}\varphi(\gamma(f,{x}))d{x}.
\end{equation}

\subsection{Property \ref{p1}} Let $0<k<n$ and $\OP{X}_{n,k}(D_+)\subset \OP{X}_n(S^2)$ 
the subspace where in each $n$-tuple of points exactly $k$ points lie in $D_+$. As usual, we denote by $D_\varphi$ the defect of the quasimorphism $\varphi$.
Now,
\begin{align*} 
&|\Phi_n(gf) - \Phi_n(f)-\Phi_n(g)|\leq \sum_{k=2}^{n-2}\int\limits_{\OP{X}_{n,k}(D_+)}|\varphi(\gamma(gf,{x}))-\varphi(\gamma(g,{x}))-\varphi(\gamma(f,{x}))|dx \leq \\ 
&\sum_{k=2}^{n-2}\int\limits_{\OP{X}_{n,k}(D_+)}|\varphi(\gamma(gf,{x}))-\varphi(\gamma(g,{f(x)}))-\varphi(\gamma(f,{x}))|dx+\\
&\sum_{k=2}^{n-2}\int\limits_{\OP{X}_{n,k}(D_+)}|\varphi(\gamma(g,{f(x)}))-\varphi(\gamma(g,{x}))|dx\leq\\
&\vol(\OP{X}_n(S^2))D_\varphi+2\int\limits_{\OP{X}_n(S^2)}|\varphi(\gamma(g,{x}))|dx\leq
\vol(\OP{X}_n(S^2))D_\varphi+2(C'+D'|g|_1),
\end{align*}
where the last inequality follows immediately from \cite[Theorem 1]{BS-sphere} 
(see also \cite[Theorem 1]{BMS-SM}). By setting 
$$C:=(\vol(\OP{X}_n(S^2))D_\varphi+2C'), D:=2D'$$ 
we obtain Property \ref{p1}.

\subsection{Property \ref{p4}} Let $f\in\G$. Then as before
\begin{equation*} 
|\Phi_n(f)|\leq\sum_{k=2}^{n-2}\int\limits_{\OP{X}_{n,k}(D_+)}|\varphi(\gamma(f,{x}))|dx\leq
\int\limits_{\OP{X}_n(S^2)}|\varphi(\gamma(f,{x}))|dx\leq C'+D'|f|_1.
\end{equation*}
By setting $A:=\max\{C',D'\}$ we obtain Property \ref{p4}.

\subsection{A fragmentation result}

\begin{lma}\label{lma: frag}
Let $L$ be an embedded contractible loop or a separating loop in a surface $\Sigma$ and $A,B$ connected open sets such that $\Sigma\setminus L = A \sqcup B$. Suppose that $\phi$ is the time-one map of a compactly supported Hamiltonian isotopy such that $\phi(L) = L$. Then there is a constant $K$ depending on $L$ only, such that $d_1(\phi, fg) < K$ for some $f \in \Ham_c(A)$ and $g \in \Ham_c(B)$.
\end{lma}

\begin{proof} First $f=\phi|_L: L \to L$ is a diffeomorphism connected to the identity or not. If $\Sigma \neq S^2$ then $f$ is connected to the identity, as it has degree $1$. We can show this as follows. In this case, the two connected components of $\Sigma \setminus L$ are not isotopic to each other inside $\Sigma$ and therefore are both preserved under $\phi$. Indeed, either the two components are not diffeomorphic or they have positive genus, in which case their inclusions induce maps with distinct images on the first homology group. In particular we may consider $\phi$ as a self-map of the pair $(D,L)$ where $D$ is the closure of one connected component of $\Sigma \setminus L$. Now $\deg(f) =1$ as $\deg(\phi)=1$ and the isomorphism $H_2(D,L;\Z) \to H_1(L;\Z)$ coming from the long exact sequence of a pair is functorial.

If $\Sigma = S^2$ and $f$ is not connected to the identity, we compose $\phi$ with a rotation $\rho$ of $S^2$ by angle $\pi$ along an axis passing through $L.$ Then $\phi'=\phi\rho$ satisfies $\phi'(L) = L,$ $f'=\phi'|_L:L \to L$ isotopic to the identity and $|\phi'|_1 \leq |\phi_1|+ c$ for $c$ depending on the geometry of $S^2$ only. 

Therefore, we may assume that $f$ is isotopic to the identity. Let $\{f_t\}$ for $f_t \in \mrm{Diff}(L)$ be an isotopy with time-one map $f.$ Lift $\{f_t\}$ to a Hamiltonian isotopy $\{\psi_t\}$ of $T^*L$ with Hamiltonian $F(t,x)$ $1$-homogeneous in the momentum variable and vanishing on the zero section $L$ for all $t.$ Identify $L$ with $S^1$ and $T^*L$ with $\R \times S^1.$ Let $\chi: \R \to [0,1]$ 
be a cutoff function such that $\chi(x) = 1$ for $|x| \leq 1/3,$ $\mrm{supp}(\chi) \subset (-1,1)$ and $|\chi'| \leq 2$ everywhere. For $\delta > 0$ set $\chi_\delta(x) = \chi(x/\delta).$ Then $\chi_\delta(x) = 1$ for  $|x| \leq \delta/3,$ $\mrm{supp}(\chi_\delta) \subset (-\delta,\delta)$ and $|\chi'| \leq 2\delta^{-1}$ everywhere. Finally, consider $\chi_\delta$ as a function on $\R \times S^1 \cong T^*L,$ set 
$$G(t,x) = F(t,x) \chi_\delta(x),$$ 
and let $\{\psi^\delta_t\}$ be the flow of $G$. Then, for all $t,$ $\psi^\delta_t|_L = f_t$ and $\mrm{supp}(\psi^\delta_t) \subset (-\delta, \delta) \times S^1.$ The key property of this construction is the following.


\begin{clm}
As $\delta \to 0,$ $\ell_1(\{\psi_t^\delta\}) \to 0.$
\end{clm}

Let us prove this statement. As $F(t,x) = 0$ for all $x = (0, q)$ and all $t,$ there exists a function $H(t,x)$ such that $F(t,p,q) = p H(t, p, q).$ Hence $\chi_\delta(p) F(t, p,q) = \chi_\delta(p) p H(t, p,q)$. 
Therefore, \[d(\chi_\delta(p) F(t, p,q)) = \chi'_\delta(p) p H(t,p,q) dp + \chi_\delta(p) H(t,p,q) dp + \chi_\delta(p)p dH_t(p,q).\] Now as $\delta \to 0,$ taking the size of the support of $\chi_\delta$ into account \[\chi'_\delta(p) p H(t,p,q) dp = O(\delta^{-1} \delta) = O(1),\] \[\chi_\delta(p) H(t,p,q) dp = O(1),\] \[\chi_\delta(p)p dH_t(p,q) = O(\delta)\] and therefore \[\ell_1(\{\psi^\delta_t\}) = 2 \delta (O(1)+O(1)+O(\delta)) = O(\delta).\] This finishes the proof.

Finally, identify a neighborhood $U$ of $L$ in $\Sigma$ with a neighborhood $V$ of $L$ in $T^*L.$ 
Fix $\varepsilon > 0$ and consider $\delta$ sufficiently small so that $\mrm{supp}(\psi^\delta_t)$ is contained in $V$ for all 
$t$ and $\{\psi_t^\delta\}$ corresponds to an isotopy $\{\phi_t^\delta\}$ on $\Sigma$ supported in $U$ such that $$|\phi^\delta_1| \leq \ell(\{\phi^\delta_t\}) < \varepsilon.$$ 
Then $(\phi^\delta_1)^{-1} \phi = f g$ for some $f, g \in \cl G$ 
such that $\mrm{supp}(f) \subset \overline{A}$ and $\mrm{supp}(g) \subset \overline{B}$ 
and $d_1(\phi, (\phi^\delta_1)^{-1} \phi) = |\phi^\delta_1| < \varepsilon.$ Finally, 
rescaling $f, g$ a little by the Liouville flow in a neighborhood of $\overline{A}$ and $\overline{B}$ 
towards the skeleta of $A, B$ we obtain $f', g'$ such that $\mrm{supp}(f') \subset {A}$ 
and $\mrm{supp}(g') \subset {B}$ and $d_1(fg, f'g')<\varepsilon.$ 
In total, we obtain 
$$d_1(\phi, f'g') < 2\varepsilon$$ 
for $\phi|_L$ isotopic to the identity. Setting $K:=c+2\varepsilon$, 
where we set $c=0$ whenever $\Sigma \neq S^2,$ we conclude the proof of the lemma for a general $\phi$.
\end{proof}

\subsection{Property \ref{p3}} Let $h\in\G$ such that $h(L)=L$. First, we show that
there exists a constant $C_1\in\mathbb{R}$ such that for every $f\in\Ham_c(D_+)$ and $g\in\Ham_c(D_-)$ we have
$|\Phi_n(fg)|\leq C_1$.

Let $\{f_t\}$ be an isotopy in $\Ham_c(D_+)$ between $f_0=\Id$ and $f_1=f$, 
and $\{g_t\}$ be an isotopy in $\Ham_c(D_-)$ between $g_0=\Id$ and $g_1=g.$ Then we have the following braid identity
$$\gamma(\{f_t * g_t\},x)=\gamma(\{f_t\},g(x))\cdot\gamma(\{g_t\},x).$$
Since $\{f_t\}\in\Ham_c(D_+)$ and $\{g_t\}\in\Ham_c(D_-)$, then
for every $1<k<n-1$ and each $x\in\OP{X}_{n,k}(D_+)$ the braids $\gamma(\{f_t\},g(x))$ and $\gamma(\{g_t\},x)$ are reducible. Hence 
$\varphi(\gamma(f,g(x)))=0$ and $\varphi(\gamma(g,x))=0$. It follows that 
$$|\varphi(\gamma(fg,x))|\leq D_\varphi.$$ 
Setting $C_1:=D_\varphi\vol(\OP{X}_n(S^2)),$ we obtain
\begin{equation}\label{eq: fg}|\Phi_n(fg)|\leq\sum_{k=2}^{n-2}\int\limits_{\OP{X}_{n,k}(D_+)}|\varphi(\gamma(fg,{x}))|dx\leq C_1.\end{equation}

We proceed with the proof of Property \ref{p3}. It follows from Lemma \ref{lma: frag} that there 
is a constant $K$ depending on $L$ only, such that $d_1(h, fg)<K$ for some $f\in\Ham_c(D_+)$ and $g\in\Ham_c(D_-)$.
Now,
\begin{align*} 
&|\Phi_n(h) - \Phi_n(fg)|\leq|\Phi_n(h) - \Phi_n(fg)-\Phi_n(h(fg)^{-1})|+|\Phi_n(h(fg)^{-1})|\leq\\
&C+D|h(fg)^{-1}|_{1}+|\Phi_n(h(fg)^{-1})|\leq C+D|h(fg)^{-1}|_{1}+(A+1)|h(fg)^{-1}|_{1}=\\
&C+(D+A+1)|h(fg)^{-1}|_{1}\leq C+(D+A+1)K,
\end{align*}
where the second inequality follows from Property \ref{p1}, and the third inequality follows from Property \ref{p4}. Hence by \eqref{eq: fg}
$$|\Phi_n(h)|\leq | \Phi_n(fg)|+C+(D+A+1)K\leq C_1+C+(D+A+1)K.$$
We conclude the proof by denoting $B:=C_1+C+(D+A+1)K$.

\subsection{Property \ref{p2}} First, we show that for each $f\in\G$ the sequence $\{\Phi_n(f^k)/k\}_{k=1}^\infty$ is bounded, 
which in turn implies that $\ol{\Phi}_n(f)$ is well-defined. It follows from \cite[Lemma 2.1]{2019-BM1} that there exists a constant $K_1$ such that 
for almost every $x\in\OP{X}_n(S^2)$ we have $|\varphi(\gamma(f,x))|\leq K_1$. Cocycle condition \ref{eq:cocycle} implies that 
$$\varphi(\gamma(f^k,x))=\varphi(\gamma(f,f^{k-1}(x))\cdot\ldots\cdot\gamma(f,x)).$$ 
Hence 
$$|\varphi(\gamma(f^k,x))|/k\leq K_1+D_\varphi.$$
The above inequality yields
\begin{align*} 
&|\Phi_n(f^k)/k|\leq\sum_{k=2}^{n-2}\int\limits_{\OP{X}_{n,k}(D_+)}|\varphi(\gamma(f^k,x))|/k\thinspace dx\leq\\
&\int\limits_{\OP{X}_n(S^2)}|\varphi(\gamma(f^k,x))|/k\thinspace dx\leq (K_1+D_\varphi)\vol(\OP{X}_n(S^2)),
\end{align*} 
which implies that $\ol{\Phi}_n(f)$ is well-defined.

Now we prove that there exists $h\in\G$ such that $\ol{\Phi}_n(h)\neq 0$. For the simplicity we prove this fact for $n=4$.
Let $0\neq \varphi\in Q_{\OP{BF}}(B_4(S^2))$ and $\beta\in P_4(S^2)<B_4(S^2)$ such that $\varphi(\beta)=b\neq 0$. By the Ishida construction
\cite{Ishida} (see also \cite{2019-BM1, 2015-IJM}), there exist four embedded discs $D_1, D_2, D_3, D_4$ in $S^2$ and $f\in\G$ such that:
\begin{itemize}
\item[\textbullet] $D_1,D_2\subset D_+$, $D_1\cap D_2=\emptyset$, $D_3,D_4\subset D_-$, $D_3\cap D_4=\emptyset$, $a_i:=\area(D_i)$.
\item[\textbullet] For each $k\in\mathbb{Z}$, each permutation $\sigma\in S_4$ and each $x=(x_1,x_2,x_3,x_4)\in \OP{X}_4(S^2)$ such that 
$x_i\in D_{\sigma(i)}$ we have $\varphi(\gamma(f^k,x))=kb$.
\end{itemize}
By definition $\ol{\Phi}_4(f)=\lim_{p\to\infty}\Phi_4(f^{k_p})/k_p$ for some increasing sequence $\{k_p\}_{p=1}^\infty$. Thus
\begin{align*} 
&\ol{\Phi}_4(f)=\lim_{p\to\infty}\Phi_4(f^{k_p})/k_p=\lim_{p\to\infty}\int\limits_{\OP{X}_{4,2}(D_+)}\frac{\varphi(\gamma(f^{k_p},x))}{k_p}\thinspace dx=\\
&\lim_{p\to\infty}\left(\int\limits_{\OP{X}_4(D_1\cup\ldots\cup D_4)}\frac{\varphi(\gamma(f^{k_p},x))}{k_p}\thinspace dx +
\int\limits_{\OP{X}_{4,2}(D_+)\setminus \OP{X}_4(D_1\cup\ldots\cup D_4)}\frac{\varphi(\gamma(f^{k_p},x))}{k_p}\thinspace dx\right),
\end{align*}
where $\OP{X}_4(D_1\cup\ldots\cup D_4)$ consists of those 4-tuples $x=(x_1,x_2,x_3,x_4)\in \OP{X}_4(S^2)$ where each $x_i$ lies in one of $D_j$.
Note that construction of $f$ yields
\begin{align} 
&\lim_{p\to\infty}\int\limits_{\OP{X}_4(D_1\cup\ldots\cup D_4)}\frac{\varphi(\gamma(f^{k_p},x))}{k_p}\thinspace dx=4!(ba_1a_2a_3a_4+
c_{1133}a_1^2a_3^2+c_{1144}a_1^2a_4^2+\label{eq:polynomial}\\
&c_{1134}a_1^2a_3a_4+c_{1233}a_1a_2a_3^2+c_{1244}a_1a_2a_4^2+c_{2233}a_2^2a_3^2+c_{2234}a_2^2a_3a_4+c_{2244}a_2^2a_4^2),\nonumber
\end{align}
where $c_{ijkl}:=\varphi(\gamma(f, (x_i,x_j,x_k,x_l)))$ such that $x_i\in D_i, x_j\in D_j, x_k\in D_k, x_l\in D_l$. 

The expression in \ref{eq:polynomial} is a homogeneous polynomial $P(a_1,a_2,a_3,a_4)$ in variables $a_1,a_2,a_3,a_4$,
i.e., $P(ra_1,ra_2,ra_3,ra_4)=r^4P(a_1,a_2,a_3,a_4)$. Without loss of generality we assume that $\area(S^2)=1$. 
Since polynomial $P$ is non-trivial, there exist $a_1,a_2,a_3,a_4\in\mathbb{R}$ 
and $0\neq c\in\mathbb{R}$ such that $a:=a_1+a_2+a_3+a_4<1$ and $P(a_1,a_2,a_3,a_4)=c$.
Again by  Ishida construction, for each $0<r<\frac{1}{a}$ 
there exist four embedded discs $D_{1,r}, D_{2,r}, D_{3,r}, D_{4,r}$ in $S^2$ and $f_r\in\G$ such that:
\begin{itemize}
\item[\textbullet] $D_{1,r}, D_{2,r}\subset D_+$, 
$D_{1,r}\cap D_{2,r}=\emptyset$, $D_{3,r}, D_{4,r}\subset D_-$, $D_{3,r}\cap D_{4,r}=\emptyset$, $ra_i=\area(D_{i,r})$.
\item[\textbullet] For each $k\in\mathbb{Z}$, each permutation $\sigma\in S_4$ and each $x=(x_1,x_2,x_3,x_4)\in \OP{X}_4(S^2)$ such that 
$x_i\in D_{\sigma(i),r}$ we have $\varphi(\gamma(f_r^k,x))=kb$.
\end{itemize}
We obtain
$$\ol{\Phi}_4(f_r)=\lim_{p\to\infty}\Phi_4(f_r^{k_{r,p}})/k_{r,p}=r^4c+
\lim_{p\to\infty}\int\limits_{\OP{X}_{4,2}(D_+)\setminus \OP{X}_4(D_{1,r}\cup\ldots\cup D_{4,r})}
\frac{\varphi(\gamma(f_r^{k_{r,p}},x))}{k_{r,p}}\thinspace dx.$$

There exists $d\in\mathbb{N}$ and $A_i\in P_n(S^2)$ the standard Artin generator for each $1\leq i\leq d$ such that
$\beta=A_1\cdot\ldots\cdot A_d$. Again, by Ishida construction 
$$f_r=f_{1,r}\circ\ldots\circ f_{d,r},$$ 
where each $f_{i,r}\in\G$ is a Morse autonomous diffeomorphism (on its support). 
Moreover, for almost every $x\in\OP{X}_4(S^2)$ we have 
$$\gamma(f_{i,r},x)=\alpha'_{f_{i,r},x}\delta_{f_{i,r},x}\alpha''_{f_{i,r},x}$$
such that $\delta_{f_{i,r},x}$ is a commuting product of reducible braids 
and the length of braids $\alpha'_{f_{i,r},x},\alpha''_{f_{i,r},x}$ 
is universally bounded by a constant which does not depend on $r$ and $x$, 
see e.g. \cite[proof of Theorem 4.5]{BrandenburskyKnots}. 
It follows that there exists $K_1>0$ such that for almost every 
$x\in \OP{X}_4(S^2)$ and every $0<r<\frac{1}{a}$ we have
$$|\varphi(\gamma(f_{i,r},x))|\leq K_1.$$
It follows that 
$$\frac{\left|\varphi(\gamma(f_r^{k_{r,p}},x))\right|}{k_{r,p}}\leq d(K_1+D_\varphi).$$
Thus 
$$\left|\lim_{p\to\infty}\int\limits_{\OP{X}_{4,2}(D_+)\setminus \OP{X}_4(D_{1,r}\cup\ldots\cup D_{4,r})}
\frac{\varphi(\gamma(f_r^{k_{r,p}},x))}{k_{r,p}}\thinspace dx\right|
\leq d(K_1+D_\varphi)\vol(\OP{X}_{4,2}(D_+)\setminus \OP{X}_4(D_{1,r}\cup\ldots\cup D_{4,r})).$$

Since $\lim_{r\to\frac{1}{a}}\vol(\OP{X}_{4,2}(D_+)\setminus \OP{X}_4(D_{1,r}\cup\ldots\cup D_{4,r}))=0$, we obtain
$$\lim_{r\to\frac{1}{a}}\ol{\Phi}_4(f_r)=\frac{c}{a^4}\neq 0.$$
It follows that there exists $r$ such that $\ol{\Phi}_4(f_r)\neq 0$. By setting $h:=f_r$ we finish the proof of Property \ref{p2}.

\section{General loops and surfaces}\label{sec: other} 

In this section we prove Theorem \ref{thm: generalized}.
Let $L$ be a separating curve in $\Sigma$. 
Note that in order to prove Theorem \ref{thm: generalized} it is enough to construct a functional 
$$\Phi_{\Sigma, n}:\G_\Sigma\to\mathbb{R}$$ 
which satisfies properties \ref{p1}, \ref{p2}, \ref{p3} and \ref{p4} in the case of 
$\G_\Sigma$.

The construction of $\Phi_{\Sigma, n}$ is very similar to the construction of $\Phi_n$
presented in Subsection \ref{ssec: paramorphism}. 
Let $n\in\mathbb{N}$ and let $\OP{C}_n(\Sigma)$ 
be the configuration space of all unordered $n$-tuples of pairwise
distinct points in $\Sigma$. Recall the Birman Push map (see e.g. \cite{Farb-Margalit}):
$$
\OP{Push}\colon B_n(\Sigma)\to\OP{MCG}(\Sigma,n),
$$
where $B_n(\Sigma)=\pi_1(\OP{C}_n(\Sigma),z)$ is the surface braid group on $n$ strands
and $\OP{MCG}(\Sigma,n)$ is the mapping class group of the $n$-punctured surface $\Sigma$. 
This map is injective when $\Sigma$ is a hyperbolic surface. 
In case when $\Sigma$ is a torus $T^2$ this map has a kernel
which equals to the center $Z(B_n(T^2))\cong\mathbb{Z}^2$. The braid $\alpha$ 
is called \emph{reducible} if $\OP{Push}(\alpha)$ is a reducible mapping class.
We denote by $Q_{\OP{BF}}(B_n(\Sigma))$ the space of homogeneous quasimorphisms
on $B_n(\Sigma)$ which vanish on reducible braids.
It follows from the celebrated paper by Bestvina and Fujiwara \cite{BestvinaFujiwara} that
the space $Q_{\OP{BF}}(B_n(\Sigma))$ is infinite dimensional for $n>1$, see \cite[Section 4.A.]{2019-BM1}.
Note that when $n=1$ and $\Sigma$ is hyperbolic, then the space $Q_{\OP{BF}}(\pi_1(\Sigma, z))$ is infinite dimensional as well,
see \cite[Section 4.A.]{2019-BM1}.

Let $\{f_t\}$ be an isotopy in $\G_\Sigma$ from the identity 
to $f\in\G_\Sigma$. For $x,y\in\Sigma$ we define a loop
$\gamma_{x,y}\colon [0,1]\to \Sigma$ by
$$
\gamma_{x,y}(t):=
\begin{cases}
\alpha_{3t} &\text{ for } t\in \left [0,\frac13\right ]\\
f_{3t-1}(x) &\text{ for } t\in \left [\frac13,\frac23\right ]\\
\beta_{3t-2} & \text{ for } t\in \left [\frac23,1\right ],
\end{cases}
$$
where $\alpha_t$ is a shortest path on $\Sigma$ from $y$ to $x$,
and $\beta_t$ is a shortest path on $\Sigma$ from $f(x)$ to $y$.

Let $\OP{X}_n(\Sigma)$ be the configuration space of all ordered $n$-tuples
of pairwise distinct points in $\Sigma$. Its fundamental group
$\pi_1(\OP{X}_n(\Sigma),z)$ is identified with the surface pure braid group $P_n(\Sigma),$ where $z=(z_1,\ldots,z_n)$ in $\OP{X}_n(\Sigma)$ is a base point.
For almost every $x=(x_1,\ldots,x_n)\in \OP{X}_n(\Sigma)$ the
$n$-tuple of loops $(\gamma_{x_1,z_1},\ldots,\gamma_{x_n,z_n})$ is
a based loop in the configuration space $\OP{X}_n(\Sigma)$.
Let $\gamma(\{f_t\},x)\in P_n(\Sigma)=\pi_1(\OP{X}_n(\Sigma),z)$
be the element represented by this loop, and let
$\varphi\colon P_n(\Sigma)\to \B R$ be a homogeneous quasimorphism.
If $\Sigma$ is a hyperbolic surface then the center $Z(P_n(\Sigma))$ is trivial, and hence 
the braid $\gamma(\{f_t\},x))$ does not depend on the choice of the isotopy $\{f_t\}$.
If $\Sigma=T^2$, and $\varphi$ vanishes on $Z(B_n(T^2))\cong\mathbb{Z}^2$ then the number
$\varphi(\gamma(\{f_t\},x))$ does not depend on the choice of the isotopy $\{f_t\}.$ In both cases from now on $\varphi(\gamma(\{f_t\},x))$
is denoted by $\varphi(\gamma(f,x))$. Similarly to the spherical case, for $f,g\in\G_\Sigma$ we have the following cocycle condition:
\begin{equation}\label{eq:cocycle Sigma}
\varphi(\gamma(gf,x))=\varphi(\gamma(g,f(x))\cdot\gamma(f,x)).
\end{equation}

\textbf{Case 1}. Let $L$ be an $a$-equator in $\Sigma$ and  $D_L$ be an open disk in $\Sigma$ whose boundary is $L$. 
Let $n>3$ and denote by $\OP{X}_n(D_L\cup (\Sigma\setminus D_L)^\circ)\subset\OP{X}_n(\Sigma)$ the subspace 
consisting of those tuples of $n$ points in $\Sigma$, such that either all points in such a tuple lie in $D_L$, 
or all points in such a tuple lie in $(\Sigma\setminus D_L)^\circ$, 
or exactly $n-1$ points lie in $D_L$ and one point in $(\Sigma\setminus D_L)^\circ$, 
or exactly $n-1$ points lie in $(\Sigma\setminus D_L)^\circ$ and one point in $D_L$. 
Let $\varphi\in Q_{\OP{BF}}(B_n(\Sigma))$ a nontrivial homogeneous quasimophism. 
We define $\Phi_{\Sigma, n}: \cl G_\Sigma \to \R$ as follows:
\begin{equation}\label{eq:paramorphism surfaces}
\Phi_{\Sigma, n}(f):=
\int\limits_{\OP{X}_n(\Sigma)}\varphi(\gamma(f,{x}))d{x}\quad-\int\limits_{\OP{X}_n(D_L\cup (\Sigma\setminus D_L)^\circ)}\varphi(\gamma(f,{x}))d{x}.
\end{equation}

Now, the proof of Properties \ref{p1}, \ref{p3} and \ref{p4} is identical to the proof of these properties in the spherical case. The proof 
of Property \ref{p2} is very similar to the spherical case as well. 
The key idea, similar to the spherical case, is to use the Ishida construction for $\G_\Sigma$ presented in 
\cite{2019-BM1}, see also \cite{2015-IJM, 2018-CCM}. We leave the details for the interested reader 
and obtain a proof of Case 1.

\textbf{Case 2}. Let $L$ be a separating curve in $\Sigma$ such that $\Sigma\setminus L=\Sigma_1\sqcup\Sigma_2$, where $\Sigma_i$ is a non-contractible surface for $i=1,2$. Note that in this case $\Sigma$ must be a hyperbolic surface.
Let $\varphi\in Q_{\OP{BF}}(\pi_1(\Sigma, z))$ a nontrivial homogeneous quasimophism, and 
$\Phi_\Sigma: \cl G_\Sigma \to \R$ be the Polterovich quasimorphism (see \cite{PolterovichDynamicsGroups}) defined as follows:
\begin{equation}
\Phi_\Sigma(f):=
\int\limits_{\Sigma}\varphi(\gamma(f,{x}))d{x}.
\end{equation}

Now, the proof of Property \ref{p1} is immediate since $\Phi_\Sigma$ is a quasimorphism. The proof 
of Property \ref{p2} is presented in \cite{2019-BM1}, see also \cite{2015-IJM, 2018-CCM}. Property \ref{p4}
follows immediately from \cite[Theorem 1]{BMS-SM}. 

Let us discuss the proof of Property \ref{p3}. Note that each 
$\varphi\in Q_{\OP{BF}}(\pi_1(\Sigma, z))$ vanishes on every $\alpha\in \pi_1(\Sigma, z)$ which is represented by a loop supported 
in either $\Sigma_1$, or in $\Sigma_2$ since each such $\alpha$ is reducible. 
Let $f\in\cl G_\Sigma$ supported in $\Sigma_1$ or in $\Sigma_2$. Then for each $x\in\Sigma$ we have 
$\gamma(f,x)=\alpha_{f,x}\beta_{f,x}$, 
where the word length of $\alpha_{f,x}$ is bounded by a constant $C$ which depends only on the geometry of $\Sigma$,
and $\beta_{f,x}$ is a reducible element in $\pi_1(\Sigma, z)$. 
It follows that there exists $C_1$ which depends on the geometry of $\Sigma$,
such that for every $f\in\cl G_{\Sigma_1}$ and $g\in\cl G_{\Sigma_2}$ we have $|\Phi_\Sigma(fg)|\leq C_1$.
Now we proceed exactly as in the spherical case and obtain the proof of Property \ref{p3}, and hence of Case 2 as well.

\begin{rmk}\label{rmk: compact}
Theorem \ref{thm: generalized} is also true for an orientable surface of any genus with non-trivial boundary. Moreover it applies to ``diameters" in the surface, defined as embedded paths with endpoints on the boundary, which are contractible relative to the boundary (in the case of the disk and Hofer's metric the geometry of this space was studied in \cite{Khanevsky-diam}). More specifically, we fix the area of the disk component of the complement of the path, and look at the space of diameters coinciding with a given fixed diameter in a neighborhood of the boundary. The proof goes along the same lines and is left to the reader.   
\end{rmk}

\begin{rmk}
Finally, we expect Proposition \ref{prop: non-deg} and Theorem \ref{thm: generalized} to generalize to the case where $L \subset \Sigma$ is a smooth embedded {\em non-separating} curve. It would be interesting to investigate this direction in the future.
\end{rmk}



\bibliographystyle{plainurl}
\bibliography{lpeq-refs}
\noindent\\

\end{document}